\documentclass[12pt]{amsart}
\usepackage{graphicx}
\usepackage{amsmath,amscd,amssymb,amsthm,amsfonts,bbm}
\usepackage{geometry}
\usepackage{epsfig,epstopdf}
\usepackage{mathrsfs}
\usepackage{cite}
\usepackage[colorlinks=true,anchorcolor=blue,filecolor=blue,linkcolor=blue,urlcolor=blue,citecolor=blue]{hyperref}
\usepackage{color}

\theoremstyle{plain}
\newtheorem{thm}{Theorem}[section]
\newtheorem{coro}[thm]{Corollary}
\newtheorem{lemma}[thm]{Lemma}
\theoremstyle{definition}

\theoremstyle{plain}
\newtheorem{rmk}[thm]{Remark}
\theoremstyle{definition}

\numberwithin{equation}{section}
\numberwithin{figure}{section}
\def \d{\delta}
\def \deg{\text{deg}}
\def \dn{\text{d}}
\def \down{\text{down}}
\def \ex{\text{ex}}
\def\p{\partial}
\def \q{\mathfrak{q}}
\def \R{\mathbb{R}}
\def \B{{\sf{B}}}
\def \un{\text{u}}
\def \up{\text{up}}
\def \sp{\text{spex}}
\def \W{{\sf{W}}}
\def \WW{\mathcal{W}}
\usepackage{amsaddr}

\begin{document}
\title[Spectral radius of complexes]{Signless Laplacian spectral radius of simplicial complexes without $r$-dimensional wheels}
\author[H.-Z. Zhang and Y.-Z. Fan]{Huan-Zhi Zhang and Yi-Zheng Fan*}
\address{\small Center for Pure Mathematics, School of Mathematical Sciences, \\
 Anhui University, Hefei 230601, P. R. China}
\email{fanyz@ahu.edu.cn, zhanghz@stu.ahu.edu.cn}
\thanks{*The corresponding author.
Supported by National Natural Science Foundation of China (No. 12331012).}
	
	

\begin{abstract}
An $r$-dimensional wheel is defined as the join of an $(r-2)$-simplex and a cycle.
In this paper, we study the maximum signless Laplacian spectral radius of  $n$-vertex $r$-dimensional pure simplicial complexes that contain no $r$-dimensional wheels.
For sufficiently large $n$, we determine the extremal complexes that attain this maximum.
Our result generalizes the corresponding extremal results of signless Laplacian on graphs and provides a spectral anlogue of a theorem of Brown, Erd\H{o}s and S\'os on the maximum number of facets of simplicial complexes in the case $r=2$.
\end{abstract}
	
\maketitle

\section{Introduction}
\subsection{Tur\'an problems}
The \emph{Tur\'an number} of a graph $F$, denoted by $\ex(n,F)$, is the maximum number of edges in an $n$-vertex graph that does not contain $F$ as a subgraph.
In 1941, Tur\'an \cite{Turan1941} proved that for any integer $r \geq 2$, the Tur\'an number of the complete graph $K_{r+1}$ on $r+1$ vertices is uniquely attained at the balanced complete $r$-partite graph $T_{r}(n)$ on $n$ vertices, which implies that  $\ex(n, K_{r+1}) = e(T_{r}(n))$, where $e(G)$ denotes the number of edges of a graph $G$.
This result initiated extensive research on Tur\'an numbers for various forbidden graphs.
A cornerstone of the field is the Erd\H{o}s-Stone-Simonovits theorem \cite{ES1966}, which states that for any graph $F$,
\[\ex(n, F) = \left(1 - \frac{1}{\chi(F) - 1} + o(1)\right) \binom{n}{2},\]
where $\chi(F)$ denotes the chromatic number of $F$.
When $F$ is bipartite, i.e., $\chi(F)=2$, the theorem only yields $\ex(n,F)=o(n^2)$.
Therefore, determining the exact value or order of magnitude of the Turán number for bipartite graphs is of great significance.
The study of such a case dates back to the work of  K\H{o}v\'ari, S\'os and Tur\'an \cite{KST1954} in 1954, who proved that for the complete bipartite graph $ K_{s,t}$ with $s \leq t$, $\ex(n, K_{s,t}) = O(n^{2-1/s}).$
A graph $G$ is said to be \emph{$(r+1)$-color-critical} if $\chi(G)=r+1$, and $\chi(G-e)=r$ for some edge $e$ of $G$, where $G-e$ denotes the subgraph of $G$ by removing the edge $e$.
For any $(r+1)$-color-critical graph $G$ on $n$ vertices, Simonovits \cite{Simon1968} proved that for sufficiently large $n$, $\ex(n, G) = e(T_{r}(n))$ and $T_{r}(n)$ is the unique extremal graph.
For further results on the Tur\'an number of graphs, see \cite{BN2010,E1964,Fur1996,FS2013,K2011}.

Let $K_n, C_n$ denote the complete graph and the cycle on $n$ vertices respectively, and let $G \vee H$ be the \emph{join} of two vertex-disjoint graphs $G$ and $H$, which is obtained from $G \cup H$ by adding all possible edges between the vertices of $G$ and the vertices of $H$.
The \emph{wheel graph} on $n$ vertices is defined as $W_n=K_1 \vee C_{n-1}$.
For even wheels, $W_{2k}$ is  $4$-color-critical, and hence, by Simonovits's result, for sufficiently large $n$, $\ex(n, W_{2k}) = e(T_{3}(n))$.
Dzido \cite{D2013} improved this result to $n \ge 6k - 10$ for $k \ge 3$.
For odd wheels, Dzido and Jastrz\c{e}bski \cite{DJ2018} determined the exact Tur\'an numbers for $W_5$ and $W_7$, and presented a lower bound for $W_{2k+1}$;
Yuan \cite{Y2021} established an exact value of $\ex(n, W_{2k+1})$ for $k \geq 3$ and sufficiently large $n$.
For disjoint union of wheels, Xiao and Zamora \cite{XiaoZ} determined the Tur\'an number and extremal graphs.

The Tur\'an problem of graphs was naturally generalized to hypergraphs.
For an $r$-uniform hypergraph $F$, the Tur\'an number  $\ex_r(n, F)$ is the maximum number of edges in an $r$-uniform hypergraph on $n$ vertices that does not contain $F$ as a subhypergraph.
However, the Tur\'an problem for hypergraphs  presents considerable difficulty.
A prime example is the $3$-uniform hypergraph $\Delta_4^3$ on $4$ vertices with all possible $3$-subsets as edges.
Tur\'an \cite{Turan1941} posed a conjecture on the exact value of $\ex_3(n , \Delta_4^3)$, which is still unknown to this day.

In fact, we can view hypergraphs from the viewpoint of simplicial complexes and consider the homeomorphic Tur\'an problems for simplicial complexes.
For a family $\mathcal{F}$ of $r$-dimensional complexes, let $\ex(n,\mathcal{F})$ denote the maximum number of $r$-faces in an $r$-dimensional complex on $n$ vertices that does not contain any complex in $\mathcal{F}$ as a subcomplex.
Brown, Erd\H{o}s and S\'os \cite{SEB1973} proved
$ \ex(n, S^2)=\Theta(n^{\frac{5}{2}})$, where $S^r$ denotes an $r$-dimensional sphere.
More recently, Newman and Pavelka \cite{NP2024} provided a conditional lower bound and also an upper bound for $\ex(n,S^r)$.
The asymptotics for homeomorphs of fixed orientable surfaces in $2$-dimensional complexes were established by Kupavskii et al. \cite{KPTZ2022},
and those for homeomorphs of fixed non-orientable surfaces were established by Sankar \cite{San2024}.
In general, homeomorphic Tur\'an problems for arbitrary complexes have been investigated by Keevash et al. \cite{KLNS2021} and  Long et al. \cite{LNY2022}.
A systematic investigation of such extremal problems was recently initiated by Conlon, Piga and Sch\"ulke \cite{CPS2023} and was further studied by Axenovich et al. \cite{AGLP2025}.

\subsection{Spectral Tur\'an problems}
The spectral version of the Tur\'an problems, known as the spectral Tur\'an problems, asks for the maximum spectral radius of all $F$-free graphs and has received a lot of attention.
Usually, the spectral radius of a graph $G$ refers to  the spectral radius of the adjacency matrix $A(G)$ of $G$, denoted by $\lambda(G)$.
Let $\sp(n,F)$ denote the maximum spectral radius among all $n$-vertex graphs that do not contain $F$ as a subgraph.
In 2007, Nikiforov \cite{Niki2007} proved that $\sp(n,K_{r+1}) = \lambda(T_r(n))$,
and $T_{r}(n)$ is the unique extremal graph.
Subsequently, Nikiforov \cite{Niki2009} obtained a spectral analog of the Erd\H{o}s-Stone-Simonovits theorem.
The spectral Tur\'an problem is closely related to  the Tur\'an problem by the following connection: $\lambda(G) \ge {2e(G)}/{n}$ for a graph $G$ on $n$ vertices, which implies that $ \ex(n,G) \le {n \cdot \sp(n,G)}/{2}.$
A wealth of literature has been devoted to this field; see \cite{LB2023,LLF2022,Niki2011}.

Zhao, Huang, and Lin \cite{ZHL2021} obtained the maximum spectral radius among all wheel-free graphs on $n$ vertices and determined the extremal graphs that have different structures depending on $n$ modulo $4$, together with an additional case of $n=7$.
For a specified odd wheel $W_{2k+1}$ for all $k \ge 2$ and $k \notin \{4,5\}$, Cioab\u{a}, Desai and Tait \cite{CioDT} determined the structure of the spectral extremal graphs.
Later, Zhang \cite{Zhang} settled the remaining case $k \in \{4,5\}$.
Luo, Ni and Dong \cite{LuoND} determined the structure of the spectral extremal graphs for disjoint odd wheels.

Alternatively, people use the spectral radius of the signless Laplacian matrix to study the spectral Tur\'an problem.
Let $q(G)$ denote the spectral radius of the signless Laplacian matrix $Q(G)$ of $G$.
By the relation
$ q(G) \ge {4e(G)}/{n}$ for an $n$-vertex graph $G$,
spectral Tur\'an results with respect to the signless Laplacian spectral radius also yield results for Tur\'an problems.
For a wheel-free graph $G$ on $n \ge 4$ vertices,
Zhao, Huang, and Lin \cite{ZHL2021} proved that
\begin{equation}\label{Graph-sign}
q(G) \le q\bigl(K_2 \vee (n-2)K_1\bigr),
\end{equation}
with equality if and only if $G= K_2 \vee (n-2)K_1$, where $(n-2)K_1$ denotes the disjoint union of $(n-2)$ isolated vertices.
For further related research, refer to the survey \cite{Z2011}.

The spectral Tur\'an problems for simplicial complexes were initiated by Fan and She \cite{FS2025}.
Let $\q_{r-1}(K)$ denote the spectral radius of the $(r-1)$-up signless Laplacian operators of an $r$-dimensional pure simplicial complex $K$, referred to as the \emph{signless Laplacian spectral radius of $K$}.
Note that if $r=1$, then $K$ is a graph and $\q_0(K)$ is consistent with the definition $q(G)$ above.
Fan and She \cite{FS2025} obtained the maximum signless Laplacian spectral radius of $r$-dimensional pure simplicial complexes on $n$ vertices without $r$-dimensional holes and determined the extremal complexes.
In a subsequent paper, She, Fan, and Song \cite{SFS2026} established an asymptotic formula for the maximum signless Laplacian spectral radius of $r$-dimensional pure simplicial complexes on $n$ vertices with the second Betti number $t$ for $1 \le t \le n-3$ and determined the extremal complexes when $t$ equals $1$ or $2$.

Let $\Delta_{n}$ denote a simplex on $n$ vertices, which can also be considered a simplicial complex on $n$ vertices with all possible subsets as faces.
An \emph{$r$-dimensional wheel} on $n$ vertices, denoted by $\W_n^r$, is defined as the join of a simplex $\Delta_{r-1}$ and a cycle $C_{n-r+1}$, written as
$\W_n^r=\Delta_{r-1} \star C_{n-r+1}$.
Note that the join $\star$ here is different from the join $\vee$ defined on graphs; see Section \ref{Sec-2} for details.
Let $\WW^r$ be the set of all $r$-dimensional wheels.
S\'os, Erd\H{o}s and Brown \cite{SEB1973} proved that
\begin{equation}\label{EBS}
 \ex(n, \WW^2)=\Theta( n^2 /3 ).
\end{equation}

In this paper, we will provide a spectral version of Erd\H{o}s-Brown-S\'os's result on wheel-free complexes.
Our main result is as follows, where an \emph{$r$-dimensional book} on $n$ vertices is denoted and defined by  $\B_n^{r}=\Delta_{r} \star (n-r)K_1$.

\begin{thm}\label{main}
 Let $K$ be a pure $r$-dimensional complex on $n$ vertices without $r$-dimensional wheels.
Then
\begin{equation}\label{upp-main}
\q_{r-1}(K) \le n.
\end{equation}
Furthermore, if $K$ is $r$-down path connected and $n \ge (r^2-1)(r-1)!+r$, then the equality holds in \eqref{upp-main} if and only if $K \cong \B_n^{r}$.
\end{thm}

\begin{rmk} \rm
(1) Let $K$ be a simplicial complex, and let $K^{(1)}$ be the $1$-skeleton of $K$.
When $r=2$, then $(\W_n^2)^{(1)} = K_1 \vee C_{n-1}$ is the wheel graph, and $(\B_n^2)^{(1)}=K_2 \vee (n-2)K_1$ is the book graph.
So, Theorem \ref{main} yields a higher dimensional version of Zhao-Huang-Lin's result \eqref{Graph-sign}.

When $r=1$, then $\W_n^1=C_n$ is a cycle, and $\B_n^1=K_1 \vee (n-1)K_1$ is a star graph.
It is known for an acylic graph $K$ on $n$ vertices, $\q_0(K) \le \q_0(\B_n^1)$, with equality if and only if $K \cong \B_n^1$.
So, Theorem \ref{main} generalizes the above result on graphs to simplicial complexes.

(2) The requirement for the lower bound of $n$ in Theorem \ref{main} is necessary.
By Corollary \ref{r+3} and Lemma \ref{eig-uni}, an $r$-dimensional cocycle complex on $r+3$ vertices associated with a cycle of length greater than $4$ is not a book but satisfies the equality of \eqref{upp-main}.
In particular, for $r=3$, consider the $3$-dimensional complex $K$ on $7$ vertices whose $3$-faces are $$0125,0136,0145,0156,0345,1235,1236,1346,2345,3456,$$
where a face $\{a,b,c,d\}$ is briefly written as $abcd$.
It is clear that $K$ contains no $3$-dimensional wheels, and is not a $3$-dimensional book.
However, $K$ is $1$-neighbor uniform, and hence by Lemma \ref{eig-uni}, $\q_2(K)=7$, satisfying the equality of \eqref{upp-main}.

(3) Let $\sp(n,\WW^r)$ denote the maximum spectral radius of the $(r-1)$-up signless Laplacian of $n$-vertex $\WW^r$-free pure $r$-dimensional complexes.
Then, Theorem \ref{main} gives
$$ \sp(n, \WW^r)=n,$$
a spectral analogue of Erd\H{o}s-Brown-S\'os's result when $r=2$.
\end{rmk}
   	
\section{Preliminaries}\label{Sec-2}
In this section, we introduce some basic notions related to simplicial complexes, Laplacians, and signless Laplacians.
	
\subsection{ Simplicial complexes and Laplace operators}
A \emph{hypergraph} $H = (V(H), E(H))$ consists of a vertex set $V(H)$ and an edge set $E(H)$, where each edge of $E(H)$ is a subset of $V(H)$.
The hypergraph $H$ is called \emph{$r$-uniform}
if each edge has exactly $r$ vertices and is called \emph{linear} if any two edges intersect in at most one vertex.
Clearly, simple graphs are $2$-uniform linear hypergraphs.

Let $V$ be a finite set.
An \emph{abstract simplicial complex} (simply called a \emph{complex}) $K$ over $V$ is a collection of subsets of $V$ that is closed under inclusion (i.e., if $F \in K$ and $F' \subseteq F$, then $F' \in K$).
A \emph{facet} of $K$ is a face that is maximal under inclusion (no other face in $K$ contains it).
$K$ is \emph{pure} if all the facets have the same size.
Pure complexes can be viewed as uniform hypergraphs with all facets as hyperedges.

An \emph{$r$-face} or an \emph{$r$-simplex} of $K$ is an element of $K$ with cardinality $r+1$.
The \emph{dimension} of an $r$-face is $r$, and the dimension of $K$ is the maximum dimension of all faces of $K$,  denoted dim $K$.
Sometimes, we simply call an $r$-dimensional complex an \emph{$r$-complex}.
The \emph{degree} of an $r$-face $F$ is the number of $(r+1)$-faces in $K$ that contain $F$, denoted by $\deg(F)$.
For two $(i+1)$-simplices of a complex $K$ sharing an $i$-face, we refer to them as \emph{$i$-down neighbors}.
For two $i$-simplices of $K$ which are faces of an $(i+1)$-simplex, we call them \emph{$(i+1)$-up neighbors}.
For a face $F$ of $K$, we denote by $N^{\dn}(F)$ the set of down neighbors and by $N^{\un}(F)$ the set of up neighbors of $F$, respectively.
The $\emph{link}$ of a face $F$ in $K$ is defined as the subcomplex of $K$ consisting of all faces $F'$ that are disjoint from  $F$  and whose union $F \cup F'$ is also a face of $K$.
The \emph{join} of two complexes $K$ and $K'$ with disjoint vertex sets is denoted and defined by $K \star K'=\{F \cup F': F \in K, F' \in K'\}$, or equivalently, it is obtained from $K \cup K'$ by adding all unions of a face of $K$ and a face of $K'$.
	
Let $S_r(K)$ denote the set of all $r$-faces of $K$.
The $p$-skeleton of $K$, written $K^{(p)}$, is the subcomplex consisting of all faces of $K$ with dimensions smaller than or equal to $p$.
The $1$-skeleton $K^{(1)}$ is a graph with vertex set $S_0(K)$ and edge set $S_1(K)$.
We say $K$ is \emph{connected} if the graph $K^{(1)}$ is connected.
An \emph{$r$-up path} is a sequence of $r$-faces $F_1,F_2,\ldots,F_m$ such that $F_s$ and $F_t$ are $(r+1)$-up neighbors if and only if $|s-t|=1$.
Similarly, an \emph{$r$-down path} is a sequence of $r$-faces $F'_1,F'_2,\ldots,F'_m$ such that $F'_s$ and $F'_t$ are $(r-1)$-down neighbors if and only if $|s-t|=1$.
The complex $K$ is called \emph{$r$-up path connected} if any two $i$-faces are connected by an $r$-up path; and $K$ is called \emph{$r$-down path connected} if any two $r$-faces are connected by an $r$-down path.
For a pure $r$-complex, it is $(r-1)$-up path connected if and only if it is $r$-down path connected, as any $(r-1)$-face of $K$ is contained in an $r$-face.

	
A face $F$ is \emph{oriented} if we choose an ordering on its vertices and write $[F]$.
Two orderings of the vertices are said to determine the same orientation if there is an even permutation transforming one ordering into the other.
If the permutation is odd, then the orientations are opposite.
	
The \emph{$i$-th chain group} $C_i(K,\R)$ of $K$ with coefficients in $\R$ is a vector space over the field $\R$ with a basis $B_i(K,\R)=\{[F]\mid F\in S_i(K)\}$.
The \emph{cochain groups} $C^i(K,\R) $ are defined as duals of the chain groups, i.e., $C^i(K,\R) = \text{Hom}(C_i(K,\R),\R)$, which are generated by the dual basis consisting of $[F]^*$ for all $F \in S_i(K)$, where
\[
[F]^*([F]) = 1,\quad [F]^*([F']) = 0 \quad \text{for } F' \neq F.
\]
The functions $[F]^*$ are called \emph{elementary cochains}.
	
The \emph{$i$-th boundary map} $\partial_{i}:C_{i}(K,\mathbb{R})\to C_{i-1}(K,\mathbb{R})$ is defined by \[\partial_i[v_{0},\cdots,v_{i}]=\sum_{j=0}^{i}(-1)^{j}[v_{0},\cdots,\hat{v}_{j},\cdots ,v_{i}],\]
where $\hat{v}_{j}$ indicates that the vertex $v_j$ is omitted.
The \emph{$i$-th coboundary map} $\d_{i}:C^{i}(K,\mathbb{R})\to C^{i+1}(K,\mathbb{R})$ is defined by $\d_{i}f=f\p_{i+1}$, that is,
\[(\d_{i}f)[v_{0},\cdots,v_{i+1}]=\sum_{j=0}^{i+1}(-1)^{j}f([v_{0},\cdots,\hat{v}_{j},\cdots ,v_{i+1}]).\]
	
Endowing $C^i(K,\mathbb{R})$ and $C^{i+1}(K,\mathbb{R})$ with positive definite inner products, respectively, we obtain
the adjoint $\delta_i^* : C^{i+1}(K,\mathbb{R}) \to C^i(K,\mathbb{R})$ of $\delta_i$, which is defined by
\[
\langle \delta_i f_1, f_2\rangle_{C^{i+1}(K,\R)} = \langle f_1, \delta_i^* f_2\rangle_{C^i(K,\R)}
\]
for all $f_1 \in C^i(K,\mathbb{R})$, $f_2 \in C^{i+1}(K,\mathbb{R})$.
	
Three Laplace operators \cite{HJ2013a} on $C_i(K,\R)$ are defined as
\begin{equation}\label{Laplace}
L_i^{\up}(K)=\d_{i}^*\d_{i}, ~ L_i^{\down}(K)=\d_{i-1}\d_{i-1}^*, ~
L_i(K)=\d_{i}^*\d_{i} + \d_{i-1}\d_{i-1}^*
\end{equation}
which are referred to as the \emph{$i$-th up Laplace operator}, \emph{$i$-th down Laplace operator}, and \emph{$i$-th Laplace operator} of $K$, respectively.
Moreover, it follows directly from the definition that $L_i^{\up}(K),  L_i^{\down}(K)$ and $L_i(K)$ are self-adjoint, non-negative, and compact operators.
Eckmann \cite{E1944} proved the discrete version of Hodge theorem, namely, the kernel of $L_i(K)$ is isomorphic to the $i$-th homology group of $K$.

In the definition of Laplacian operators, for each $i$, if the inner product in $C^i(K,\R)$ is defined such that the elementary cochains are pairwise orthonormal, namely $([F]^*,[F]^*)=1$ and $([F]^*,[F']^*)=0$ if $F \ne F'$,
then the above three Laplacians are called \emph{combinatorial Laplacians}.
In addition, one can refer to \cite{SWF2025} for normalized Laplacians eigenvalues.

\subsection{Signless Laplace operators}
The signless Laplacian of a simplicial complex was introduced in \cite{KO2020}, and it was systematically studied with many interesting applications \cite{KL2014,Lub2014}.
It was further extended to signless $1$-Laplacian for investigating combinatorial property of a complex \cite{LZ2020}.
In addition, one can refer to \cite{FWW2025} for the relation between Laplacian eigenvalues and signless Laplacian eigenvalues.

Along with the line of the definition of Laplace operators, we can define the signless Laplace operators similarly by omitting the orientation.
Consider the vector space $D_i(K,\R)$ over $\R$ generated by all $i$-faces of $K$.
The \emph{$i$-th signless boundary map}
 $|\p_i|: D_i(K,\R) \to D_{i-1}(K,\R)$ is defined by
$$ |\p_i|\{v_0,\ldots,v_i\}=\sum_{j=0}^i \{v_0,\ldots,\hat{v}_j, \ldots, v_i\},$$
where $\hat{v}_j$ denotes the omission of the vertex $v_j$.
Let $D^i(K,\R)$ be the dual space of $D_i(K,\R)$ that is generated by dual basis $\{F^*: F \in S_i(K)\}$.
The \emph{$i$-th signless coboundary map} $|\d_i|: D^i(K,\R) \to D^{i+1}(K,\R)$ is defined by $ |\d_i| f =f |\p_{i+1}|$ for each $f \in D^i(K,\R)$.

For each $i$, endow $D^i(K,\R)$ with a positive inner product that makes
 the dual basis $\{F^*: F \in S_i(K)\}$ orthonormal. (One can define a general positive inner product over $D^i(K,\R)$.)
Let $|\d_i^*|:D^{i+1}(K,\R) \to D^i(K,\R)$ be the adjoint of $|\d_i|$,
which satisfies
$$\langle |\d_i| f, g \rangle_{D^{i+1}(K,\R)}=\langle f, |\d_i^*| g \rangle_{D^{i}(K,\R)}$$
for all $f \in D^i(K,\R)$ and $g \in D^{i+1}(K,\R)$.
The \emph{$i$-th up signless Laplace operator} and \emph{$i$-th down signless Laplace operator} of $K$ are respectively defined by
\begin{equation}\label{SLaplace}
Q_i^{\up}(K)=|\d_{i}^*| |\d_{i}|, ~ Q_i^{\down}(K)=|\d_{i-1}| |\d_{i-1}^*|.
\end{equation}

For each $f \in D^i(K,\R)$, by definition, we have the following expressions (or see \cite{FS2025}).
\begin{equation}\label{Q-up}
(Q_i^{\up}(K)f)(F)=\deg (F)  f(F) + \sum_{F' \in N^\un(F)} f(F');
\end{equation}
\begin{equation}\label{Q-dn}
(Q_i^{\down}(K)f)(F)=|F| f(F) + \sum_{F' \in N^\dn(F)} f(F').
\end{equation}

Since $Q_i^{\up}(K) = (Q_{i+1}^{\down}(K))^\top$ (in matrix forms),
 $Q_i^{\up}(K)$ and $(Q_{i+1}^{\down}(K))$ share the same nonzero eigenvalues.
 So, it suffices to deal with $Q_i^{\up}(K)$.
	
\section{Higher dimensional wheels}
\subsection{Cocycle complexs}
Let $V$ be a set with $(r+3)$ vertices, and let $C$ be a cycle whose vertex set is a subset of $V$.
An $r$-dimensional complex on the vertex set $V$ is called a \emph{cocycle complex associated with a cycle $C$}, denoted by $K_C$, if its $r$-faces consist of the complements of the edges of $C$ in the set $V$, namely,
$$ S_r(K_C)=\{V \setminus e: e \in E(C) \}.$$
A cocycle complex $K_C$ is in fact a complementary complex in the sense of the complement of faces; see \cite[Defintion 4.1]{DR2002}.
For convenience, denote $F_e:=V \setminus e$, the face associated with the edge $e$ of $C$.
For any two edges $e,e'\in E(C)$, $F_e$ and $F_{e'}$ are down neighbors if and only if $|e \cap e'|=1$, i.e., $e$ is incident to $e'$.
If the vertex set of $C$ is $V$ itself, then all $r$-faces $F_e$ with $e \in E(C)$ constitute a tight cycle.
Sudakov and Tomon \cite{ST2022} presented an upper bound for the extremal number of tight cycles.
In particular, if $r=2$, all $2$-faces of the tight cycle are a triangulation of M\"obius strip; see Fig. \ref{mobius}.
For more on tight cycles, refer to \cite{KPTZ2022,ST2022}.

\begin{figure}[htbp]
\centering
\includegraphics[scale=.8]{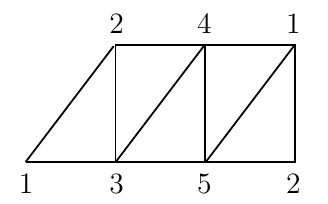}
\caption{\small A triangulation of M\"obius strip}\label{mobius}
\end{figure}

For brevity, denote by $C=v_1 v_2 \ldots v_n$ a cycle on vertices $v_1, \ldots,v_n$ with edges $\{v_i,v_{i+1}\}$ for $i=1,\ldots,n$, where $v_{n+1}=v_1$.

\begin{lemma}\label{cycle-com}
Let $K_C$ be an $r$-dimensional cocycle complex associated with a cycle $C$ of length $\ell$.
The following results hold.

{\rm (1)}  $K_C$ is $r$-path connected.

{\rm (2)} If $\ell=3$, then $K_C$ is the  book $\B_{r+3}^r$.

{\rm (3)} If $\ell=4$, then $K_C$ is the wheel $\W_{r+3}^r$.

{\rm (4)} If $ 5 \le \ell \le r+3$, then $K_C$ is $\WW^r$-free.
\end{lemma}

\begin{proof}
(1)  Let $F$ and $F'$ be two distinct $r$-faces of $K_C$.
By construction, there exist edges $e,e'\in E(C)$ such that $F=F_e$, $F'=F_{e'}$.
Since $C$ is a cycle, there exists a path in $C$ with edges $e=e_0,e_1,\dots,e_m=e'$, such that $e_{i-1}$ is incident with $e_i$ for $i=1, \ldots, m$.
So, the faces $F_{e_{i-1}}$ and $F_{e_i}$ are down neighbors for $i=1,\ldots,m$, which implies that $F_{e_0}, \ldots, F_{e_m}$ forms an $r$-path connecting $F=F_{e_0}$ and $F'=F_{e_m}$.
Hence, $K_C$ is $r$-path connected.

(2) If $l=3$, then $K_C$ has exactly three $r$-faces that share a common $(r-1)$ face, which implies that $K_C$ is the  book $\B_{r+3}^r$.

(3) If $l=4$, then $K_C$ has exactly four $r$-faces that share a common $(r-2)$ face, say $E$, and the link of $E$ is the cycle $C$.
So $K_C$ is the wheel $\W_{r+3}^r$.

(4) Assume to the contrary that $K_C$ contains an $r$-dimensional wheel $W$.
Suppose that $W$ is a join of an $(r-2)$ face, say $E$, with a cycle $C'$.
Since $V(C') \subseteq V \setminus E$, $C'$ has $3$ or $4$ vertices.
If $C'$ has $3$ vertices, say $v_1,v_2,v_3$, then there is an additional vertex  outside $W$, say $v_4$.
By the definition of cocycle complex, $W$ contains the following $r$-faces:
$ F_{\{v_3,v_4\}}, F_{\{v_1,v_4\}}, F_{\{v_2,v_4\}}$
Thus, the vertex $v_4$ has degree at $3$ in the cycle $C$; a contradiction.

If $C'=v_1v_2v_3v_4$, then also by definition, $K_C$ contains four $r$-faces $F_e$ for all $e \in E(C')$.
However, for the above $e$, $e \in E(C)$, and hence $E(C') \subset E(C)$, which yields a contradiction as $C$ has length at least $5$.
\end{proof}

\subsection{$1$-neighbor  uniformity}
Let $K$ be an $r$-dimensional complex. For an $r$-face $F$ and $u\in V(K)\backslash F$, define
\[N^\dn(F,u):=\left\{G\cup\left\{u\right\}\in S_r(K):G\in \p F\right\},\]
where $\p F$ denotes the set of all $(r-1)$-faces in the boundary of $F$.
An $r$-dimensional complex $K$ is called \emph{$t$-neighbor uniform} if for any $r$-face $F$ and any $v \in V(K)\setminus F$, $|N^{\dn}(F,v)|=t$, namely, each vertex outside $F$ contributes $t$ down neighbors to the face $F$.

Here we focus on the $1$-neighbor uniformity of $r$-dimensional complexes.
We find that the cocycle complexes play an important role on $1$-neighbor uniform $r$-complexes.

\begin{lemma}\label{1-uni}
Let $K_C$ be an $r$-dimensional cocylce complex associated with a cycle $C$.
Then $K_C$ is $1$-neighbor uniform.
\end{lemma}

\begin{proof}
Let $F$ be an $r$-face of $K_C$.
Then, $F=F_e$ for some edge $e\in E(C)$.
Let $u \in V(K_C) \setminus F_e =e$.
If $F_{e'} \in N^\dn(F_e,u)$, then $e'$ is incident to $e$, and $e'$ does not contain $u$.
Since $e$ is incident to exactly $2$ edges in $C$, one of which contains the vertex $u$, $e'$ is the unique edge of $C$ incident to $e$ that does not contain $u$,
which implies that $|N^\dn(F_e,u)|=1$.
So, $K_C$ is $1$-neighbor uniform.
\end{proof}

\begin{rmk}
In general, a  $1$-neighbor uniform complex may not be $r$-path connected.
For example, let $K$ be the complement complex of two disjoint cycles $C_1$ and $C_2$.
Then $K$ is a union of two cocycle complexes, namely, $K_{C_1}$ and $K_{C_2}$.
For any two edges $e_1 \in E(C_1)$ and $e_2 \in E(C_2)$, since there exist no paths connecting $e_1$ and $e_2$ in $C_1 \cup C_2$, the face $F_{e_1}$ cannot connect to $F_{e_2}$ by an $r$-path.
\end{rmk}

\begin{lemma}\label{1-uniform}
Let $K$ be an $r$-down path connected pure $r$-dimensional complex on a vertex set $V$ of size $r+3$.
Then $K$ is $1$-neighbor uniform if and only if $K$ is cocycle complex associated a cycle $C$ of length $\ell$ with $3 \le \ell \le r+3$.
\end{lemma}

\begin{proof}
By Lemma \ref{1-uni}, it suffices to prove the necessity.
Suppose that $K$ is $1$-neighbor uniform.
Then any $r$-face of $K$ can be written as $F_e:=V\setminus e$ for some $2$-subset $e$ of $V$.
Let $G$ be a graph on the vertex set $V$ whose edges are those $e$ such that $F_e$ is an $r$-face of $K$.
Recall that $F_e$ and $F_{e'}$ are down neighbors if and only if $e$ intersects $e'$ in exactly one vertex.
Now, take an edge $e=\{u,v\}$ of $G$.
Then $e$ corresponds to an $r$-face $F_e$ of $K$, and $V \setminus F_e=e$.
By the $1$-neighbor uniformity, $|N^\dn(F_e,u)|=1$.
Let $F_{e'}$ be the unique down neighbor in $N^\dn(F_e,u)$.
Then, $e'$ is unique edge of $G$ incident with $e$ via the vertex $v$.
So, $v$ has degree $2$ in the graph $G$ arising from the contribution of $e$ and $e'$.
Similarly, $u$ also has degree $2$ in $G$.
Therefore, each endpoint of an edge of $G$ has degree $2$.

Since $K$ is $r$-down path connected, any two $r$-faces are connected by an $r$-down path.
This implies that any two edges of $G$ are connected by a path.
So, $G$ is a union of a connected subgraph $G'$ with some isolated vertices.
By the above discussion, $G'$ is $2$-regular, and hence $G'$ is a cycle $C$ of length $\ell$ with $3 \le \ell \le r+3$.
Thus, $K$ is the cocycle complex $K_C$ associated with $C$.
\end{proof}

\begin{coro}\label{r+3}
Let $K$ be an $r$-path connected $r$-dimensional complex on $r+3$ vertices without $\WW^r$.
If $K$ is $1$-neighbor uniform, then $K \cong \B_{r+3}^{r}$, or $K$ is an $r$-dimensional cocycle complex associated with a cycle of length at least $5$.
\end{coro}

\begin{proof}
By Lemma \ref{1-uniform}, $K$ is a cocycle complex $K_C$ associated with a cycle of length $\ell$, where $3 \le \ell \le r+3$.
Since $K$ contains no wheels, by Lemma \ref{cycle-com}, we have $\ell \ne 4$.
So, $K \cong \B_{r+3}^{r}$ when $\ell =3$ or $K$ is the cocycle complex $K_C$ when $\ell \ge 5$.
\end{proof}

\begin{lemma}\label{eig-uni}
Let $K$ be a pure $r$-dimensional complex on $n$ vertices.
If $K$ is  $1$-neighbor uniform, then
 $\q_{r-1}(K)  =n. $
\end{lemma}

\begin{proof}
Suppose that $|N^{\dn}(F,u)| = 1$, for any  $F \in S_r(K)$ and any $u \in V(K) \backslash F$.
Then each $r$-face has exactly $(n-(r+1))$ down neighbors, one for each vertex outside it.
A direct computation shows that the all‑ones vector $\mathbf{1}$ satisfies $$Q_{r}^{\down}(K) \mathbf{1} = n \mathbf{1},$$
because for any $r$-face $F$, by \eqref{Q-dn},
$$ (Q_{r}^{\down}(K)\mathbf{1})(F) = |F|\cdot 1 + \sum_{F'\in N^\dn(F)} 1 = (r+1) + (n-(r+1)) = n. $$
Hence $n$ is the largest eigenvalue of $Q_{r}^{\down}(K)$ by Perron-Frobenius theorem, implying that $n$ is also the largest eigenvalue of $Q_{r-1}^{\up}(K)$.
The result follows.
\end{proof}

Recall that the book $\B_n^r=\Delta_r \star (n-r)K_1$, which is $1$-neighbor uniform.
By Lemma \ref{1-uni}, any cocycle complex is $1$-neighbor uniform.
So, by Lemma  \ref{eig-uni}, we immediately get the following result.

\begin{coro}\label{two-rad}
Let $\B_n^r$ be the $r$-dimensional book on $n$ verties, and let $K_C$ be an $r$-dimensional cocylce complex associated with a cycle $C$.
Then $$\q_{r-1}(\B_n^r)=n, ~ \q_{r-1}(K_C)=r+3.$$
\end{coro}

\subsection{Proof of main result}
\begin{lemma}\label{downeibor}
Let $K$ be an $r$-dimensional complex without $\WW^r$.
Then for any $r$-face $F$ of $K$ and any vertex $u \in V(K) \backslash F$,
$|N^{\dn}(F,u)|\le 1$.
\end{lemma}

\begin{proof}
Assume to the contrary that there exist an $r$-face $F=\{v_0,\dots,v_r\}$ and a vertex $u \in V(K) \backslash F$ such that  $|N^{\dn}(F,u)|\geq 2$.
Then, $F$ contains at least two vertices, say $v_i$ and $v_j$, such that  $F_i:=(F\setminus\{v_i\}) \cup \{u\}$ and $F_j:=(F\setminus\{v_j\})\cup \{u\}$ are both $r$-faces of $K$.
Thus, $F,F_i,F_j$ form an $r$-dimensional wheel, which is a join of the $(r-2)$-face $F\setminus \{v_i, v_j\}$ with the cycle $v_i v_j u$; a contradiction.
\end{proof}

\begin{thm}\label{main_1}
Let $K$ be a pure $r$-dimensional complex on $n$ vertices without $\WW^r$.
Then
\begin{equation}\label{upp}
\q_{r-1}(K) \le n.
\end{equation}
Furthermore, if $K$ is $r$-down path connected, then the equality holds if and only if $K$ is $1$-neighbor uniform.
\end{thm}
	
\begin{proof}
Observe that $\q_{r-1}(K)$ is also the spectral radius of $Q_{r}^{\down}(K)$.
Let $f$ be the Perron vector of $Q_{r}^{\down}(K)$ associated with the eigenvalue $\q_{r-1}(K)$.
By normalization, let $F_0$ be the $r$-face of $K$ such that
$$ f(F_0):= \max \{f(F): F \in S_r(K)\} =1.$$
Since $K$ is $\WW^r$-free, by Lemma \ref{downeibor}, for any $u \in V(K) \backslash F_0$, we have
\begin{equation}\label{localu}
|N^{\dn}(F_0,u)|\le 1,
\end{equation}
and hence
\begin{equation}\label{local}
|N^\dn(F_0)| =\sum_{u \in V(K) \backslash F_0} |N^{\dn}(F_0,u)|\le n-r-1.
\end{equation}
By \eqref{Q-dn}, we get
\begin{equation}\label{eq-r}
\begin{split}
\q_{r-1}(K) & = \q_{r-1}(K) f(F_0) =|F_0| f(F_0) + \sum_{F' \in N^\dn(F_0)} f(F')\\
&\le  r+1 + \sum_{F' \in N^\dn (F_0)} 1 \\
&\le r+1+n-r-1=n.
\end{split}
\end{equation}

Now suppose that $K$ is $r$-down path connected.
If $\q_{r-1}(K) = n$, then by \eqref{eq-r},
$|N^\dn(F_0)| =  n-r-1$,  and  $f(F)=1$ for each $F \in N^{\dn}(F_0)$.
By \eqref{localu}, the equality $|N^\dn(F_0)| =  n-r-1$ implies that  $|N^{\dn}(F_0,u)| = 1$ for all $u \in V(K) \backslash F_0$.
Applying the above discussion on any down neighbor $F \in N^{\dn}(F_0)$, we will get $f(F')=1$  for all $F' \in N^{\dn}(F)$.
Repeating the above discussion,  as $K$ is $r$-down path connected, we will have $f(F)=1$ for all $F \in S_r(K)$.
Consequently, by \eqref{localu}, for any  $F \in S_r(K)$ and any $u \in V(K) \backslash F$, we have $|N^{\dn}(F,u)| = 1$.
So, $K$ is $1$-neighbor uniform.
Conversely, if $K$ is $1$-neighbor uniform, then, by Lemma \ref{eig-uni}, $\q_{r-1}(K)=n$.
\end{proof}

Combining Theorem \ref{main_1} and Corollary \ref{r+3}, we now can prove the main result of this paper.

\begin{proof}[\bf Proof of Theorem \ref{main}]
By Theorem \ref{main_1}, we have $\q_{r-1}(K) \le n$, which implies \eqref{upp-main} in Theorem \ref{main}.
Surely, if $K \cong \B_n^r$, by Corollary \ref{two-rad}, $\q_{r-1}(K)=n$.

Now suppose $K$ is $r$-down path connected with $\q_{r-1}(K)=n$.
By Theorem \ref{main_1},  $K$ is $1$-neighbor uniform.
Let $V(K) := \{v_1 , \dots, v_{n}\}$, and let $F_0=\{v_1,\dots,v_{r+1}\}$ be an $r$-face of $K$.
Then, by the $1$-neighbor uniformity, for any $u \in V(K) \setminus F_0$, $|N^{\dn}(F_0,u)| = 1$.
By the pigeonhole principle, there exists an $(r-1)$-face $E_0 \in \p F_0$ such that $E_0$ is contained in at least  $(n-r-1)/(r+1)$ faces except the face $F_0$, or equivalently, there exist at least  $(n-r-1)/(r+1)$ vertices $v$ outside $F_0$ such that $E_0 \cup \{v\} \in S_r(K)$.
For an $(r-1)$-face $E \in \p F_0$, let
$$N(E):= \{ v \in V(K)\setminus F_0: E \cup \{v\} \in S_r(K) \}.$$
We have two cases.

{\bf Case 1.} For any other face $E \in \p F_0 $ with $E \ne E_0$, $N(E) = \emptyset$.
This implies that $E$ is only contained in the face $F_0$ and $E_0 \cup \{v\} \in S_r(K)$ for all vertices $v$ outside $F_0$.
In this case, $K$ contains a book $\B_{n}^{r}$ including the face $F_0$.
Since $K$ is $r$-path connected and $\B_{n}^{r}$ is a subcomplex of $K$,
we have $K=\B_{n}^{r}$; otherwise, $\q_{r-1}(K) > \q_{r-1}(\B_{n}^{r})= n$; a contradiction.

{\bf Case 2.} There exists a face $E \in \p F_0$ with $E \ne E_0$ such that  $|N(E)| \geq 1$.
Let $u \in N(E)$, and $E[u]:=E \cup \{u\} \in S_r(K)$.
Surely,  $u \notin F_0 \cup N(E_0)$.
Choose a vertex $w \in N(E_0)$, and let $E_0[w]:=E_0 \cup \{w\}$.
Let $F_0[uw]:=F_0 \cup \{u\} \cup \{w\}$, an $(r+3)$-subset of $V(K)$.
Let $K[F_0[uw]]$ be the subcomplex of $K$ induced by the vertices of $F_0[uw]$.
Note that the subcomplex of $K[F_0[uw]]$ consisting of the faces $F_0$, $E_0[w]$ and $E[u]$ is $r$-down path connected.
So, we can let $K'[F_0[uw]]$ be the $r$-down path connected component of $K[F_0[uw]]$ that contains the faces $F_0$, $E_0[w]$ and $E[u]$.

Suppose $E_0=F_0 \setminus \{v_1\}$ and $E=F_0\setminus \{v_2\}$.
It is known that $K'[F_0[uw]]$ is still $1$-neighbor uniform, as each of its faces is connected to its down neighbors.
By Corollary \ref{r+3}, as $K'[F_0[uv]]$ is $\WW^r$-free and is not the book $\B_{r+3}^r$, $K'[F_0[uw]]$ is a cocycle complex $K_C$ associated with a cycle $C$ of length $\ell$, where $C$ contains the edges $\{v_2, w\}, \{w,u\}, \{u, v_1\}$ corresponding to the faces $E[u],F_0,E_0[w]$ respectively, and $5\leq \ell\leq r+3$.

By the above discussion, for each $w \in N(E_0)$, $K'[F_0[uw]]$ is a cocycle complex $K_C$ associated with a cycle $C$ of length $\ell_w$ with $5\leq \ell_w \leq r+3$.
By the pigeonhole principle, there exists a subset $N' \subseteq N(E_0)$ with size at least $(n-r-1)/((r+1)(r-1))$
such that for all vertices $w \in N'$, the complexes $K'[F_0[uw]]$ are cocycle complexes associated with cycles of same length $c$, where $5 \le c \le r+3$.

For a vertex $w \in N'$, write $K'[F_0[uw]]=K_C$, where $C$ contains the edges $\{v_2, w\}$, $\{w,u\}$, $\{u, v_1\}$, and except the vertices of these $4$ edges, $C$ contains the other $(c-4)$ vertices from the $(r-1)$-set $F_0\setminus \{v_1,v_2\}$.
So, there are $ \binom{r-1}{c-4}(c-4)!$ possible configurations for the cycle $C$.
If $n \ge (r^2-1)(r-1)!+r$, we have
\begin{equation}\label{N}
|N'| \ge \frac{n-r-1}{r^2-1} > (r-1)! \ge  \binom{r-1}{c-4}(c-4)!.
\end{equation}
This implies that there exist two distinct vertices $w_1,w_2 \in N'$ such that cocycle complexes $K'[F_0[uw_1]]=K_{C_1}$ and $K'[F_0[uw_2]]=K_{C_2}$ satisfy:
$$ C_1=v_2 w_1 u v_1 x_1 \cdots x_{c-4}, ~ C_2=v_2 u w_2 u v_1 x_1 \cdots x_{c-4},$$
where $x_1, \ldots, x_{c-4} \in F_0 \setminus \{v_1,v_2\}$.

In the complex $K_{C_1}$, we have faces $G_1:=F_0[uw_1] \setminus \{u,v_1\}$, $G_2:=(F_0[uw_1] \setminus \{v_1,x_1\})$.
While, in the complex $K_{C_2}$, we have faces $G_3:=(F_0[uw_2] \setminus \{u,v_1\}$, $G_4:=(F_0 [uw_2]\setminus \{v_1,x_1\})$.
Then, $K$ contains faces $G_1,G_2,G_3,G_4$, all sharing a common $(r-2)$-face $S:=F_0 \setminus \{v_1,x_1\}$.
Therefore, the link of $S$ contains a cycle $x_1w_1uw_2$, which implies that $K$ contains a wheel; a contradiction.
\end{proof}

\end{document}